\documentclass[12pt, reqno]{amsart}
\usepackage{amsmath, amsthm, amscd, amsfonts, amssymb, graphicx, color}
\usepackage[bookmarksnumbered, colorlinks, plainpages]{hyperref}
\hypersetup{colorlinks=true,linkcolor=red, anchorcolor=green, citecolor=cyan, urlcolor=red, filecolor=magenta, pdftoolbar=true}

\newtheorem{theorem}{Theorem}[section]
\newtheorem{lemma}[theorem]{Lemma}
\newtheorem{prop}[theorem]{Proposition}
\newtheorem{cor}[theorem]{Corollary}
\theoremstyle{definition}
\newtheorem{definition}[theorem]{Definition}
\newtheorem{example}[theorem]{Example}

\theoremstyle{remark}
\newtheorem{remark}[theorem]{\bf{Remark}}
\numberwithin{equation}{section}
\allowdisplaybreaks
\begin{document}
	\title[Numerical Radius inequalities via Orlicz function]  
 { Numerical Radius inequalities via Orlicz function}
	\author[P. Bhunia and R. K. Nayak]{ Pintu Bhunia and Raj Kumar Nayak}

	\address{(Bhunia) Department of Mathematics, Indian Institute of Science, Bengaluru 560012, Karnataka, India}
 \email {pintubhunia5206@gmail.com / pintubhunia@iisc.ac.in}
	\address{(Nayak) Department of Mathematics, GKCIET, Malda, India}
	\email{rajkumar@gkciet.ac.in / rajkumarju51@gmail.com}

 \thanks{ The first author, Dr. Pintu Bhunia would like to thank SERB, Govt. of India for the financial support in the form of National Post Doctoral Fellowship (File No. PDF/2022/000325) under the mentorship of Prof. Apoorva Khare}
 
	\subjclass[2010]{47A12, 47A30, 15A60}
	\keywords{ Numerical radius, Operator norm, Orlicz function, Inequality}
	\date{}
	\maketitle

 \begin{abstract}
 Employing the Orlicz functions we extend the Buzano's inequality which is a refinement of the Cauchy-Schwarz inequality. Also using the Orlicz functions
 we obtain several numerical radius inequalities for a bounded linear operator as well as the products of operators. We deduce different new upper bounds for the numerical radius. It is shown that
\begin{eqnarray*}
		    {w(T)} \leq  \sqrt[n]{ \log \left[ \frac{1}{2^{n-1}} e^{w(T^n)} + \left( 1-\frac{1}{2^{n-1}}\right) e^{\|T\|^n}\right]} &\leq& \|T\| \quad \forall n=2,3,4, \ldots
		\end{eqnarray*} 
  where $w(T)$ and $\|T\|$ denote the numerical radius and the operator norm of a bounded linear operator $T$, respectively.
 \end{abstract}

	\section{Introduction}
     
	\noindent	Let $\mathcal{B}({\mathcal{H}})$ represent the $C^*$-algebra comprising all bounded linear operators on a complex Hilbert space $\mathcal{H}.$  The absolute value of $T,$ is defined as $|T| = (T^*T)^{\frac{1}{2}},$ where $T^*$ denotes the Hilbert adjoint of the operator $T.$ The numerical range of $T \in \mathcal{B}(\mathcal{H})$ is denoted by $W(T)$, is the image of the unit sphere of $\mathcal{H}$ under the mapping $x \rightarrow \langle Tx, x \rangle.$ The numerical radius and the operator norm of an operator  $T \in \mathcal{B}(\mathcal{H})$, are denoted by $w(T)$ and $\|T\|$ respectively, and are defined as 
	$w(T) = \sup_{\|x\|=1} \left|\langle Tx, x \rangle \right|$
	and  $\|T\| = \sup_{\|x\|=1} \|Tx\|.$
	It is universally acknowledged that $w(\cdot)$ constitutes a norm on 
	$\mathcal{B}(\mathcal{H})$  which is comparable to the  operator norm through the following inequality:	
		$\frac{\|T\|}{2} \leq w(T) \leq \|T\|.$
Significant progress has been made in addressing this inequality in recent years. Here, we emphasize a few of these noteworthy advancements. 
	\noindent In \cite{kstudiamath1}, Kittaneh demonstrated that 
		$w(T) \leq \frac{1}{2} \left\||T| + |T^*|\right\|.$ 
	This inequality was further generalized by El-Haddad et al. \cite{kstudiamath3}, who asserted that
		$w^{2r}(T) \leq \frac{1}{2} \left\||T|^{2r} + |T^*|^{2r}\right\|, r\geq 1$.
	In 2015,  Abu-Omar et al. \cite{abu omar}  demonstrated that 
		$w^2(T) \leq \frac{1}{4} \left\||T|^2 + |T^*| \right\| + \frac{1}{2} w(T^2).$
	In 2021, Bhunia et al. \cite{bulletin sci} proved that
		$w^2(T) \leq \frac{1}{4} \left\||T|^2 + |T^*|^2 \right\| + \frac{1}{2} w\left(|T||T^*| \right).$
	Subsequently, Dragomir \cite{dragomir} derived the following inequality for the product of two operators, demonstrating that for $T, S \in \mathcal{B}(\mathcal{H})$ and $r \geq 1$,
	\begin{equation} \label{eq dragomir}
		w^r(S^*T) \leq \frac{1}{2} \left\||T|^{2r} + |S|^{2r} \right\|.
	\end{equation}
	For further details on recent work regarding the numerical radius inequalities, readers can see \cite{filomat, Pintu laa 2024, pintu asm,  book, Kittaneh MIA, Nayak IIST, Nayak Acta, Sababheh_Moradi_Sahoo_LAMA_2024}.
	
An Orlicz function is a mathematical function used in the context of Orlicz spaces, which generalize classical Lebesgue spaces.
\begin{definition}
	An Orlicz function $\phi: [0, \infty) \rightarrow [0, \infty)$ is a continuous, convex and increasing function that satisfies the following conditions:
	\begin{eqnarray*}
	~(i)~&&	\phi(0) = 0\\
	~~(ii)~&&	\phi(t) >0~~\forall t >0\\
	~~(iii)	&&\lim_{t \rightarrow \infty} \phi(t) = \infty.
	\end{eqnarray*}
\end{definition}

\begin{example}
	Most common examples of Orlicz function include\\
	
	(i) ~$\phi(t) = t^p$ with $p \geq 1,$ which corresponds to the standard ${\ell}_p$ spaces.\\
	
	(ii) ~~$\phi(t) = e^t -1,$ which defines the exponential Orlicz space.\\
	
	(iii)~~$\phi(t)= t^p \log(1+t),$  for $p>0$ mixed power exponential function.\\
	
	(iv)~~$\phi(t) = e^{t^2} -1,$ quadratic exponential function.
	\end{example}
\noindent An Orlicz function is said to be sub-multiplicative if $\phi(t_1 t_2) \leq \phi(t_1) \phi(t_2),$ for all $t_1, t_2 \geq 0.$
	The study of numerical radius inequalities through the lens of Orlicz functions offers a fascinating exploration of operator theory within complex Hilbert spaces. By employing Orlicz functions, one can achieve sharper bounds on the numerical radius of operators, which has significant implications for both theoretical and applied mathematics.\\
	\noindent 	The renowned Young inequality asserts that for any two positive real numbers $x$ and $y$, along with $t$ lying in the interval $[0,1]$, 
	$	x^ty^{1-t} \leq tx + (1-t) y$ holds.
	For $t = {1}/{2}$, we encounter the famous arithmetic-geometric mean inequality, which stipulates that for any two positive real numbers $x$ and $y$, the inequality 
		$\sqrt{xy} \leq \frac{x+y}{2}$ holds.
	
	In this paper, we obtain the numerical radius inequalities of bounded linear operators involving the Orlicz functions. Choosing different Orlicz functions we deduce some new bounds for the numerical radius which refines the existing classical ones. We also deduce the existing numerical radius inequalities mentioned here.

	\section{Main Results}
	
	In this section, we will establish various generalizations and enhancements of the upper bounds for numerical radii via Orlicz functions. To achieve this, we will employ the following well-known lemmas. The first lemma arises as a consequence of the spectral theorem, combined with Jensen's inequality.
	\begin{lemma}\label{positive op}\cite{mccarthy}
		Let $T \in \mathcal{B}(\mathcal{H})$ be a positive operator and $x$ be an unit vector in $\mathcal{H}.$ Then, for $r \geq 1,$ the inequality $\langle Tx, x \rangle^r \leq \langle T^rx, x \rangle $ holds.
	\end{lemma}

The second lemma presents a norm inequality for a non-negative convex function.

	\begin{lemma}\label{convex op} \cite{aujla}
		Let $\psi$ be a non-negative convex function on $[0, \infty)$ and $T, S \in \mathcal{B}(\mathcal{H})$ be positive operators. Then  $\left\|\psi\left(\frac{T+S}{2}\right)\right\| \leq \left\|\frac{\psi(T) + \psi(S)}{2}\right\| $
		holds.
		In particular, for $r \geq 1$,  \[\left\|\left(\frac{T+S}{2} \right)^r \right\| \leq \left\|\frac{T^r+S^r}{2}\right\|. \]
	\end{lemma}

The third lemma presents a generalized formulation of the mixed Schwarz inequality.

	\begin{lemma}\label{mixed schwarz} \cite{Res}
		Let $T \in \mathcal{B}(\mathcal{H})$ and $x, y \in \mathcal{H}$. Let $\psi, \eta$ are two non-negative continuous function on $[0, \infty) $ satisfying $\psi(t) \eta(t)=t.$ Then \[|\langle Tx, y \rangle| \leq \left\| \psi(|T|)x\right\| \left\|\eta(|T^*|)y \right\|.\]
	\end{lemma}

The following outcome pertains to Buzano's renowned extension of the Cauchy-Schwarz inequality.

	\begin{lemma}\label{buzano}\cite{Buzano}
		Let $x, y, e \in \mathcal{H}$ with $\|e\|=1.$ Then  \[|\langle x, e\rangle \langle e , y\rangle| \leq \frac{1}{2}\left(\|x\|\|y\| + |\langle x, y \rangle| \right) .\]
	\end{lemma}

The next lemma represents the operator form of the classical Jensen's inequality.
	\begin{lemma}\label{jenson} \cite{Res}
		Let $T \in \mathcal{B}(\mathcal{H})$ be a self-adjoint operator whose spectrum contained in the interval $J$, and let $x \in \mathcal{H}$ be a unit vector. If $h$ is a convex function on $J$, then \[h(\langle Tx, x \rangle) \leq \langle h(T)x, x \rangle. \]
	\end{lemma}
	
The following lemma is a refinement of the Cauchy-Schwarz inequality.

	\begin{lemma}\cite[Lemma 2.7]{Nayak C-S}\label{gen cauchy}
		Let $f: (0,1) \rightarrow [0,\infty)$ be a well-defined function. Then,  \[|\langle x, y \rangle|^2 \leq \frac{f(t)}{1+f(t)} \|x\|^2 \|y\|^2  + \frac{1}{1+ f(t)} |\langle x, y \rangle | \|x\|\|y\| \leq \|x\|^2 \|y\|^2 \quad \text{for any $x, y \in \mathcal{H}$}.\] 
	\end{lemma}

First we present an enhancement of the Cauchy-Schwarz inequality through the Orlicz extension of the Buzano's inequality. The proof follows from Lemma \ref{buzano} and the convexity property of the Orlicz function.

\begin{prop} \label{Orlicz buzano}
Let $x, y,e \in \mathcal{H}$ with $\|e\|=1.$ Then for any Orlicz function $\phi,$
\[\phi\left(\left|\langle x, e \rangle \langle e, y \rangle \right|\right) \leq \frac{1}{2} \left(\phi\left(\|x\|\|y\|\right) + \phi \left(|\langle x, y \rangle |\right) \right). \]
In particular, for $\phi(t) = e^t-1$ and $e^{t^2} -1$,
\begin{eqnarray}
\left|\langle x, e \rangle \langle e, y\rangle \right| \leq \log  \left(\frac12{e^{\|x\|\|y\|} + \frac12 e^{|\langle x, y \rangle|}}\right) &\leq& \|x\|\|y\|
\end{eqnarray}
and
\begin{eqnarray}\label{0--1}
 \left|\langle x, e \rangle \langle e, y\rangle \right| \leq \sqrt{ \log  \left(\frac12{e^{\|x\|^2\|y\|^2} + \frac12 e^{|\langle x, y \rangle|^2}}\right)} &\leq& \|x\|\|y\|.
\end{eqnarray}
\end{prop}




\noindent It is given in \cite{kstudiamath1} that for any $T \in \mathcal{B}(\mathcal{H}),$ 
$w(T ) \leq \frac{1}{2} \left(\|T\| + \|T^2\|^{\frac{1}{2}}\right).$
We ontain an extended version of the above bound via the Orlicz function, the proof follows by using the convexity property of the  Orlicz function.

\begin{prop}
    Let $T \in \mathcal{B}(\mathcal{H}).$ Then for any Orlicz function $\phi$, \[\phi\left(w(T)\right) \leq \frac{1}{2} \left( \phi(\|T\|) + \phi \left(\|T^2\|^{\frac{1}{2}} \right) \right).\]
    
    In particular, for $\phi(t) = e^t-1,$ 
 \begin{eqnarray}\label{p-9}
     w(T) \leq \log \left( \frac12 { e^{\|T\| } + \frac12 e^{\|T^2\|^{\frac{1}{2}}}} \right) &\leq & \|T\|.
 \end{eqnarray} 
\end{prop}

 The bound \eqref{p-9} refines the bound $w(T) \leq \|T\|.$
Next we prove the following theorem.

	\begin{theorem} \label{th 1}
		Let $T, S \in \mathcal{B}(\mathcal{H}).$ Let $f: (0,1)\rightarrow [0,\infty)$ be any well-defined function. Then for any sub-multiplicative Orlicz function $\phi$, 
		\begin{eqnarray*}
			\phi\left(w^2(T^*S)\right) &\leq& \frac{1}{2(1+ f(t))} \phi\left(w(T^*S)\right) \left\| \phi(|T|^2) + \phi(|S|^2) \right\| \\ && + \frac{f(t)}{2(1+f(t))} \phi\left(w(|S|^2 |T|^2)\right) + \frac{f(t)}{4(1+f(t))}  \left\| \phi(|T|^4) + \phi(|S|^4) \right\| .
		\end{eqnarray*}
		\end{theorem}
	\begin{proof}
		Let $x \in \mathcal{H}$ be an unit vector. By using the convex nature of $\phi$, we get that
		\begin{eqnarray*}
		\phi\left(	\left|\langle T^*Sx, x \rangle \right|^2 \right) &=& \phi\left(\left|\langle Tx, Sx \rangle \right|^{2} \right)\\
			&\leq& \phi \left( \frac{1}{(1+ f(t))}  \|Tx\|\|Sx\| |\langle Tx, Sx \rangle| +  \frac{f(t)}{(1+ f(t))} \|Tx\|^2\|Sx\|^2 \right)\\&&\,\,\,\,\,\,\,\,\,\,\,\,\,\,\,\,\,\,\,\,\,\,\,\,\,\,\,\,\,\,\,\,\,\,\,\,\,\,\,\,\,\,\,\,\,\,\,\,\,\,\,\,\,\,\,\,\,\,\,\,\,\,\,\,\,\,\,\,\,\,\,\,\,\,\,\,\,\,\,\,\,\,\,\,\,\,\,\,\,\,\,\,\,\,\,\,\,\,\,\,\,\,\,\,\,\, (\mbox{using Lemma \ref{gen cauchy}})\\
			&\leq& \frac{1}{(1+ f(t))} \phi\left( \|Tx\| \|Sx\||\langle T^*Sx, x \rangle|\right) + \frac{f(t)}{(1+ f(t))} \phi\left(\|Tx\|^{2} \|Sx\|^{2} \right)\\
			&\leq& \frac{1}{(1+ f(t))} \phi \left(\left\langle \left(\frac{|T|^{2}+ |S|^{2}}{2} \right)x, x \right\rangle \right) \phi \left(\left|\langle T^*Sx, x \rangle \right| \right)\\&& + \frac{f(t)}{(1+ f(t))} \phi\left(\langle |T|^{2}x, x \rangle \langle x, |S|^{2}x \rangle \right)\\ 	&&\,\,\,\,\,\,\,\,\,\,\,\,\,\,\,\,\,\,\,\,\,\,\, ~(\mbox{using A.M-G.M inequality and sub-multiplicity of $\phi$})\\ 
			&\leq& \frac{1}{2(1+ f(t))}  \left\langle \left(\phi\left(|T|^{2}\right)+ \phi\left(|S|^{2}\right) \right)x, x \right\rangle \phi\left( \left|\langle T^*Sx, x \rangle \right|\right) \\&&+ \frac{f(t)}{(1+ f(t))} \phi \left( \frac{\left\||T|^{2}x\right\|\left\||S|^{2}x \right\| + |\left\langle |T|^{2}x, |S|^{2}x \right\rangle| }{2} \right)\\ &&~\,\,\,\,\,\,\,\,\,\,\,\,\,\,\,\,\,\,\,\,\,\,\,\,\,\,\,\,\,\,\,\,\,\,\,\,\,\,\,\,\,\,\,\,\,\,\,\,\,\,\,\,\,\,\,\,\,\,\,\,\,\,\,\,\,\,\,\,\,\,\,\,\,\,\,\,\,\,\,\,\,\,\,\,\,\,\,\,\,\,\,(\mbox{using Lemma \ref{buzano}})\\
			&\leq&\frac{1}{2(1+ f(t))}  \left\langle \left(\phi\left(|T|^{2}\right)+ \phi\left(|S|^{2}\right) \right)x, x \right\rangle \phi\left( \left|\langle T^*Sx, x \rangle \right|\right)  \\&&+ \frac{f(t)}{2(1+ f(t))} \phi \left( \left\langle \left(\frac{|T|^{4} + |S|^{4} }{2} \right)x, x \right\rangle \right)\\&&+ \frac{f(t)}{2(1+f(t))} \phi \left(|\left\langle \left(|S|^{2} |T|^{2}\right)x, x \right\rangle|\right)\\
		&\leq&	\frac{1}{2(1+ f(t))}  \left\langle \left(\phi\left(|T|^{2}\right)+ \phi\left(|S|^{2}\right) \right)x, x \right\rangle \phi\left( \left|\langle T^*Sx, x \rangle \right|\right)  \\&&+ \frac{f(t)}{4(1+ f(t))}  \left\langle \left(\phi\left(|T|^{4}\right) + \phi \left(|S|^{4}\right)  \right)x, x \right\rangle \\&&+ \frac{f(t)}{2(1+f(t))} \phi \left(|\left\langle \left(|S|^{2} |T|^{2}\right)x, x \right\rangle|\right)\\
		&\leq& \frac{1}{2(1+ f(t))} \phi\left(w(T^*S)\right) \left\| \phi(|T|^2) + \phi(|S|^2) \right\| \\ && + \frac{f(t)}{2(1+f(t))} \phi\left(w(|S|^2 |T|^2)\right) + \frac{f(t)}{4(1+f(t))}  \left\| \phi(|T|^4) + \phi(|S|^4) \right\|.
		\end{eqnarray*}
		By taking the supremum over all $x$ with $\|x\| =1,$ we derive the desired inequality.
	\end{proof}

	The following corollaries can be derived from Theorem \ref{th 1} by choosing different Orlicz functions.
	
	\begin{cor} 
		Let $T, S \in \mathcal{B}(\mathcal{H})$ and $f: (0,1) \rightarrow [0,\infty)$ be a well-defined function. Then for $r \geq 1,$  
	\begin{eqnarray*}
		w^{2r}(T^*S) &\leq& \frac{1}{2(1+ f(t))} w^r (T^*S)  \left\||T|^{2r} + |S|^{2r}\right\| +  \frac{f(t)}{4(1+ f(t))} \left\||T|^{4r} + |S|^{4r}\right\| \\ && + \frac{f(t)}{2(1+ f(t))}  w^r\left( |S|^{2} |T|^{2}\right)\\ &\leq& \frac{1}{2} \left\| |T|^{4r} + |S|^{4r} \right\|.
	\end{eqnarray*}	
	\end{cor}

	\begin{proof}
	By considering the Orlicz function $\phi(t) = t^r$,  $r \geq 1$ in Theorem \ref{th 1} we get the first inequality. The second inequality follows from the first inequality by using (\ref{eq dragomir}) along with Lemma \ref{convex op}.
	\end{proof}

 The above result is also given in \cite[Theorem 2.10]{Nayak C-S}.
	
	\begin{cor} \label{N-1}
		Let $T, S \in \mathcal{B}(\mathcal{H}).$ Then
		\begin{eqnarray*}
		   w^{2}(S^*T) 
		   &\leq&  \frac{1}{3}\left\||T|^{2} + |S|^{2}\right\|w(T^*S) + \frac{1}{12} \left\||T|^{4} + |S|^{4}\right\| + \frac16 w(|S|^2|T|^2)\\
		   &\leq&  \frac{1}{3}\left\||T|^{2} + |S|^{2}\right\|w(T^*S) + \frac{1}{6} \left\||T|^{4} + |S|^{4}\right\|.
		  \end{eqnarray*}
	\end{cor}

	\begin{proof}
		By choosing the Orlicz function $\phi (t) =t$  and setting $f(t) = \frac{1}{2}$ in Theorem \ref{th 1}, we get the first inequality and the second inequality 
		follows using (\ref{eq dragomir}).
	\end{proof}

The second inequality is established by Kittaneh and Moradi \cite[Theorem 1]{Kittaneh MIA}. Therefore, the inequality in Theorem \ref{th 1} (also Corollary \ref{N-1}) generalizes and improves the existing bound given in \cite[Theorem 1]{Kittaneh MIA}.

	
	\begin{cor}\label{cor-11}
		Let $T, S \in \mathcal{B}(\mathcal{H})$. For $t \in (0,1)$, 
		\begin{eqnarray*}
			w^2(T^*S) &\leq& \frac{1}{2(1+t)}\left\||T|^2 + |S|^2 \right\| w(T^*S) + \frac{t}{4(1+t)} \left\| |T|^{4} + |S|^4 \right\|\\ && + \frac{t}{2(1+t)} w\left( |S|^2|T|^2\right)\\
			&\leq& \frac{1}{2(1+t)} \left\||T|^2 + |S|^2 \right\| w(T^*S) + \frac{t}{2(1+t)} \left\||T|^4 + |S|^4 \right\|.
		\end{eqnarray*}
	\end{cor}
	\begin{proof}
		Considering the Orlicz function $\phi (t) = t$ and $f(t) =t$ in Theorem \ref{th 1}, we get the first inequality and
		the second inequality follows using (\ref{eq dragomir}).
	\end{proof}

The second inequality is given by Al-Dolat et al. \cite [Theorem 2.6] {filomat}.
Therefore, the inequality stated in Theorem \ref{th 1} (also Corollary \ref{cor-11}) extends and improves upon the inequality \cite [Theorem 2.6] {filomat}.

	If we consider the Orlicz function $\phi(t) = t^r,$ with $r \geq 1$ and  $f(t)=t$ in Theorem \ref{th 1}, we obtain the bound established by Nayak \cite[Theorem 2.16]{Nayak IIST}.
	\begin{cor} 
		Let $T, S \in \mathcal{B}(\mathcal{H}).$ Then for $t \in (0,1)$ and $r \geq 1,$
		\begin{eqnarray*}
			w^{2r}(T^*S) &\leq& \frac{1}{2(1+t)}\left\||T|^{2r} + |S|^{2r }\right\| w^r(T^*S) + \frac{t}{4(1+t)} \left\| |T|^{4r} + |S|^{4r} \right\|\\ && + \frac{t}{2(1+t)} w^r\left( |S|^{2}|T|^{2}\right).
		\end{eqnarray*}
	\end{cor}

	Our next theorem is as follows.
	
	\begin{theorem} \label{th 2}
			Let $T \in \mathcal{B}(\mathcal{H})$ and $g, h$ be two non-negative continuous functions on $[0, \infty) $ satisfying $g(t) h(t)=t$ for all $t\geq 0$. Then,  for a well-defined function $f:(0,1) \rightarrow [0,\infty)$ and for any sub-multiplicative Orlicz function $\phi$,
		\begin{eqnarray*}
		\phi\left(	w^2(T)  \right)  &\leq& \frac{f(t)}{4(1+ f(t))} \left\|\phi\left(g^4(|T|)\right) + \phi\left(h^4(|T^*|)\right)\right\|\\&& + \frac{f(t)}{2(1+ f(t))} \phi\left(w\left( h^2(|T^*|) g^2(|T|)\right)\right) \\&&+ \frac{1}{2(1+ f(t))} \phi\left(w(T)\right) \left\|\phi \left(g^2(|T|)\right) + \phi\left(h^2(|T^*|)\right) \right\|.
		\end{eqnarray*}
	\end{theorem}
	
	\begin{proof}
	Let $x \in \mathcal{H}$ be an unit vector.	Employing the convexity property of $\phi,$ we obtain
		\begin{eqnarray*}
			\phi \left(\left|\langle Tx, x \rangle\right|^2  \right) &=& \phi \left(\frac{f(t)}{1+f(t)} \left|\langle Tx, x \rangle\right|^2 + \frac{1}{1+f(t)} \left|\langle Tx, x \rangle\right|^2   \right)\\
			&\leq& \frac{f(t)}{1+f(t)} \phi \left( \left|\langle Tx, x \rangle\right|^2 \right) + \frac{1}{1+f(t)}\phi \left( \left|\langle Tx, x \rangle\right|^2 \right)\\
			&\leq& \frac{f(t)}{1+ f(t)} \phi \left( \left\langle g^2(|T|)x, x \right \rangle \left \langle h^2(|T^*|) x, x \right \rangle\right) \\ &&+ \frac{1}{1+ f(t)} \phi \left(|\langle Tx, x \rangle| \sqrt{\left\langle g^2(|T|)x, x \right \rangle \left \langle h^2(|T^*|) x, x \right \rangle} \right) \\
			&&\,\,\,\,\,\,\,\,\,\,\,\,\,\,\,\,\,\,\,\,\,\,\,\,\,\,\,\,\,\,\,\,\,\,\,\,\,\,\,\,\,\,\,\,\,\,\,\,\,\,\,\,\,\,\,\,\,\,\,\,\,\,\,\,\,\,\,\,\,\,\,\,\,\,\,\,\,\,\,\,\,\,\,\,\,\,\,\,\,\,\,\,\,\,\,~~(\mbox{using Lemma \ref{mixed schwarz}})\\
			&\leq& \frac{f(t)}{1+ f(t)} \phi \left(\frac{\left\|g^2(|T|)x \right\| \left\| h^2(|T^*|)x \right\| + |\left\langle g^2(|T|)x, h^2(|T^*|)x \right\rangle|}{2} \right) \\ && + \frac{1}{1+ f(t)} \phi \left(|\langle Tx, x \rangle| \right) \phi \left(\frac{ \left\langle g^2(|T|)x, x \right \rangle + \left \langle h^2(|T^*|) x, x \right \rangle}{2} \right)\\
			&&\,\,\,\,\,\,\,\,\,\,\,\,\,\,\,\,\,\,\,\,\,\,\,\,\,~~(\mbox{using Lemma \ref{buzano} and sub-multiplicity property of $\phi$})\\
			&\leq& \frac{f(t)}{2(1+ f(t))} \phi \left(\left\|g^2(|T|)x \right\| \left\| h^2(|T^*|)x \right\| \right)\\ && + \frac{f(t)}{2(1+ f(t))} \phi \left(|\left \langle h^2(|T^*|) g^2 (|T|) x, x \right \rangle| \right)\\ && + \frac{1}{2(1+ f(t))} \phi \left(|\langle Tx, x \rangle| \right) \left \langle \left\{ \phi \left(g^2(|T|)\right) + \phi \left(h^2(|T^*|)\right)\right\}x, x \right\rangle \\
			&\leq& \frac{f(t)}{2(1+ f(t))} \phi \left(\left\langle \left( \frac{g^4(|T|) + h^4(|T^*|)}{2} \right)x, x \right\rangle \right)\\ && + \frac{f(t)}{2(1+ f(t))} \phi \left(|\left \langle h^2(|T^*|) g^2 (|T|) x, x \right \rangle| \right)\\ &&  + \frac{1}{2(1+ f(t))} \phi \left(|\langle Tx, x \rangle| \right) \left \langle \left\{ \phi \left(g^2(|T|)\right) + \phi \left(h^2(|T^*|)\right)\right\}x, x \right\rangle 
				\end{eqnarray*}
			\begin{eqnarray*}
			&\leq& \frac{f(t)}{4(1+ f(t))} \left\langle \left(\phi\left(g^4(|T|)\right) + \phi\left(h^4(|T^*|) \right)\right) x, x  \right\rangle \\ && + \frac{f(t)}{2(1+ f(t))} \phi \left(\left \langle h^2(|T^*|) g^2 (|T|) x, x \right \rangle \right)\\ &&  + \frac{1}{2(1+ f(t))} \phi \left(|\langle Tx, x \rangle| \right) \left \langle \left\{ \phi \left(g^2(|T|)\right) + \phi \left(h^2(|T^*|)\right)\right\}x, x \right\rangle \\
			&\leq& \frac{f(t)}{4(1+ f(t))} \left\|\phi\left(g^4(|T|)\right) + \phi\left(h^4(|T^*|)\right)\right\|\\&& + \frac{f(t)}{2(1+ f(t))} \phi\left(w\left( h^2(|T^*|) g^2(|T|)\right)\right) \\&&+ \frac{1}{2(1+ f(t))} \phi\left(w(T)\right) \left\|\phi \left(g^2(|T|)\right) + \phi\left(h^2(|T^*|)\right) \right\|.
		\end{eqnarray*}
			By taking the supremum over all $x$ with $\|x\|=1$,  we obtain our required inequality.
	\end{proof}
 
	In particular, if we consider $\phi(t) = t, t\geq 0,$ then we get the following result.
 
	\begin{cor}
			Let $T \in \mathcal{B}(\mathcal{H})$ and $g, h$ be two non-negative continuous functions on $[0, \infty) $ satisfying $g(t) h(t)=t$ for all $t\geq 0.$ Then,  for a well-defined function $f:(0,1) \rightarrow [0, \infty),$ \begin{eqnarray*}
			w^2(T) &\leq& \frac{f(t)}{4(1+ f(t))} \left\|g^4(|T|) + h^4(|T^*|)\right\| + \frac{f(t)}{2(1+ f(t))}w\left( h^2(|T^*|) g^2(|T|)\right) \\&&+ \frac{1}{2(1+ f(t))} w(T) \left\|g^2(|T|) + h^2(|T^*|) \right\|.
		\end{eqnarray*}
	\end{cor}
	
This result also proved in \cite[Theorem 2.16]{Nayak C-S}.

	\begin{cor}\label{cor-22}
			Let $T \in \mathcal{B}(\mathcal{H}),$ then \begin{eqnarray*}
			w^2(T) &\leq& \frac{1}{12} \left\||T|^2 + |T^*|^2\right\| + \frac{1}{6} w\left(|T^*||T|\right) + \frac{1}{3} w(T) \left\||T| + |T^*|\right\|\\
			&\leq& \frac{1}{6} \left\||T|^2 + |T^*|^2\right\|+ \frac{1}{3} w(T) \left\||T| + |T^*|\right\|.
		\end{eqnarray*}
	\end{cor}
 
	\begin{proof}
			Taking  $\phi(t)= t$, $g(t) = h(t) = \sqrt{t}$ and $ f(t) = \frac{1}{2}$ in Theorem \ref{th 2}, we obtain
		\begin{eqnarray*}
			w^2(T) &\leq& \frac{1}{12} \left\||T|^2 + |T^*|^2\right\| + \frac{1}{6} w\left(|T^*||T|\right) + \frac{1}{3} w(T) \left\||T| + |T^*|\right\|\\
			&\leq& \frac{1}{6} \left\||T|^2 + |T^*|^2\right\|+ \frac{1}{3} w(T) \left\||T| + |T^*|\right\|~~(\mbox{using inequality (\ref{eq dragomir})}).
		\end{eqnarray*}
	\end{proof}

 The second bound proved  by Kittaneh et al.  \cite[Theorem 2]{Kittaneh MIA}.
	Hence, the inequality established in Theorem \ref{th 2} (also Corollary \ref{cor-22}) extends and improves the previous upper bound  \cite[Theorem 2]{Kittaneh MIA}.
 
Bhunia and Paul \cite[Th. 2.11]{pintu rim} proved that 
	 for any $\alpha \in [0,1]$ and for $r\geq 1,$
	\begin{eqnarray*}
		w^{2r}(T) &\leq& \frac{\alpha}{2} w^r(T^2) + \left \| \frac{\alpha}{4} |T|^{2r} + \left(1-\frac{3 \alpha}{4}\right) |T^*|^{2r} \right\|
	\end{eqnarray*}
	and 	\begin{eqnarray*}
		w^{2r}(T) &\leq& \frac{\alpha}{2} w^r(T^2) + \left \| \frac{\alpha}{4} |T^*|^{2r} + \left(1-\frac{3 \alpha}{4}\right) |T|^{2r} \right\|.
	\end{eqnarray*}
	We now generalize the inequalities via Orlicz function.

	\begin{theorem} \label{th 3}
		Let $T \in \mathcal{B}(\mathcal{H}).$ Then for any Orlicz function $\phi$ and for $\alpha \in [0,1]$,
		\begin{eqnarray*}
			\phi \left( w^2(T)\right) &\leq& \frac{\alpha}{2} \phi \left(w(T^2)\right) + \left\|\frac{\alpha}{4} \phi\left(|T|^2\right) + \left( 1- \frac{3 \alpha}{4}\right) \phi \left(|T^*|^2\right)\right\|
		\end{eqnarray*}
		and 	\begin{eqnarray*}
			\phi \left( w^2(T)\right) &\leq& \frac{\alpha}{2} \phi \left(w(T^2)\right) + w\left\|\frac{\alpha}{4} \phi\left(|T^*|^2\right) + \left( 1- \frac{3 \alpha}{4}\right) \phi \left(|T|^2\right)\right\|.
		\end{eqnarray*}
	\end{theorem}

	\begin{proof}
	Substituting $x$ with $Tx$, $y$ with $T^*x$ and $e$ with $x$,  $\|x\|=1$ in Lemma \ref{buzano}, yields
	\begin{equation} \label{buz deduction}
	\left|\langle Tx, x \rangle \right|^2 \leq \frac{1}{2} \left(|\langle T^2 x, x \rangle| + \|Tx\| \|T^*x\| \right) .\end{equation}
	
	Using the convexity property of $\phi,$ we get
	\begin{eqnarray*}
		\phi\left(|\langle Tx, x \rangle|^2\right) &=& \phi \left(\alpha |\langle Tx, x \rangle|^2 + (1-\alpha) |\langle Tx, x \rangle|^2 \right)\\ &\leq& \alpha \phi \left(|\langle Tx, x \rangle|^2\right) + (1-\alpha) \phi \left(\|T^*x\|^2 \right)\\
		&\leq& \alpha \phi \left(\frac{1}{2} \left(|\langle T^2 x, x \rangle| + \|Tx\| \|T^*x\| \right) \right)+ (1-\alpha) \phi \left(\langle |T^*|^2x, x \rangle \right)\\
		&&\,\,\,\,\,\,\,\,\,\,\,\,\,\,\,\,\,\,\,\,\,\,\,\,\,\,\,\,\,\,\,\,\,\,\,\,\,\,\,\,\,\,\,\,\,\,\,\,\,\,\,\,\,\,\,\,\,\,\,\,\,\,\,\,\,\,\,\,\,\,\,\,\,\,\,\,\,\,\,\,\,~~(\mbox{using inequality (\ref{buz deduction})})\\
		 &\leq& \frac{\alpha}{2} \phi\left(|\langle T^2x, x \rangle| \right) + \frac{\alpha}{2} \phi \left(\frac{\|Tx\|^2 + \|T^*x\|^2}{2} \right)\\ && + (1-\alpha)  \phi \left(\langle |T^*|^2x, x \rangle \right)\\ 
   &\leq&  \frac{\alpha}{2} \phi\left(|\langle T^2x, x \rangle| \right) + \frac{\alpha}{2} \phi \left(\left\langle \left(\frac{|T|^2 + |T^*|^2}{2}\right)x, x  \right\rangle \right)\\ && +  (1-\alpha)  \phi \left(\langle |T^*|^2x, x \rangle \right)\\ 
   &\leq&  \frac{\alpha}{2} \phi\left(|\langle T^2x, x \rangle| \right) + \left\langle \left(\frac{\alpha}{4} \phi \left(|T|^2 \right) + \left( 1-\frac{3 \alpha}{4}\right) \phi \left(|T^*|^2\right) \right)x, x\right\rangle\\ &\leq&  \frac{\alpha}{2} \phi \left(w(T^2)\right) + \left\|\frac{\alpha}{4} \phi\left(|T|^2\right) + \left( 1- \frac{3 \alpha}{4}\right) \phi \left(|T^*|^2\right)\right\|.
		\end{eqnarray*}
		Now taking the supremum over all $x$ with $\|x\|=1,$ we obtain the first inequality.
		For the second inequality replace $T$ by $T^*$ in the first inequality.
	\end{proof}

Next theorem reads as 
 
	\begin{theorem}\label{th 4}
			Let $T \in \mathcal{B}(\mathcal{H})$ and $g,h$ be non-negative continuous functions on $[0, \infty)$ satisfying $g(t) h(t)=t,$ $\forall$ $t\geq 0.$ Then for any Orlicz function $\phi$ and for any $\alpha \in [0,1]$,
			\[\phi\left(w^2(T) \right) \leq  \left\|\frac{\alpha}{2}\left[ \phi \left(g^4(|T|)\right) + \phi \left(h^4(|T^*|) \right)\right] + (1-\alpha) \phi \left(|T|^2\right)\right\| \]
			and 
				\[\phi\left(w^2(T) \right) \leq  \left\|\frac{\alpha}{2}\left[ \phi \left(g^4(|T^*|)\right) + \phi \left(h^4(|T|) \right)\right] + (1-\alpha) \phi \left(|T^*|^2\right)\right\| .\]
	\end{theorem}
	\begin{proof}
	Let $x \in \mathcal{H}$ with $\|x\|=1.$	Then by using the convexity property of $\phi,$ we have 
	\begin{eqnarray*}
		\phi \left(|\langle Tx, x \rangle|^2 \right) &=& \phi \left(\alpha |\langle Tx, x \rangle|^2 + (1-\alpha) |\langle Tx, x \rangle|^2\right) \\ &\leq& \alpha \phi \left(|\langle Tx, x \rangle|^2\right) + (1-\alpha) \phi \left(\|T^*x\|^2\right)\\
		&\leq& \alpha \phi \left(\langle g^2(|T|)x,x \rangle \langle h^2(|T^*|)x, x \rangle  \right) + (1-\alpha) \langle |T^*|^2 x, x \rangle\\ 
		&&\,\,\,\,\,\,\,\,\,\,\,\,\,\,\,\,\,\,\,\,\,\,\,\,\,\,\,\,\,\,\,\,\,\,\,\,\,\,\,\,\,\,\,\,\,\,\,\,\,\,\,\,\,\,\,\,\,\,\,\,\,\,\,\,\,\,\,\,\,\,~~(\mbox{using Lemma \ref{mixed schwarz}})\\
		&\leq& \alpha \phi \left(\left\langle\left(\frac{g^4(|T|) + h^4(|T^*|)}{2}\right)x, x \right\rangle \right) + (1-\alpha) \phi \left(\langle |T^*|^2x, x \rangle\right)\\
		&\leq& \frac{\alpha}{2} \left\langle \left\{\phi\left(g^4(|T|)\right) + \phi\left(h^4(|T^*|)\right)\right\}x, x  \right\rangle + (1-\alpha) \langle \phi\left(|T^*|^2\right) x, x \rangle\\ 
		&=& \left\langle \left\{\frac{\alpha}{2} \left(\phi\left(g^4(|T|)\right) + \phi\left(h^4(|T^*|)\right)\right) + (1-\alpha) \phi \left(|T^*|^2\right) \right\}x, x \right\rangle
		\\ &\leq&  \left\|\frac{\alpha}{2}\left[\phi \left(g^4(|T|)\right) + \phi \left(h^4(|T^*|) \right)\right] + (1-\alpha) \phi \left(|T^*|^2\right)\right\|.
	\end{eqnarray*}
	Now taking the supremum over all $x$ with $\|x\|=1,$ we obtain the first inequality. Second inequality can be obtained by replacing $T$ with $T^*$ in first inequality.	
	\end{proof}

In particular, taking $g(t)=h(t)=\sqrt{t}$ and $\alpha=1$ in Theorem \ref{th 4}, we obtain

\begin{cor}
    Let $T \in \mathcal{B}(\mathcal{H})$. Then for any Orlicz function $\phi$,
			\[\phi\left(w^2(T) \right) \leq \frac12 \left\|  \phi \left(|T|^2\right) + \phi \left(|T^*|^2 \right)  \right\|.\]
   \end{cor}

In \cite[Theorem 2.4]{pintu gmj}, it is given that
	\begin{eqnarray*}
			w^2(T) &\leq& \frac{1}{4}  w^2\left(|T| + i |T^*|\right) + \frac{1}{4} w \left(|T| |T^*|\right) + \frac{1}{8} \left\| |T|^2 + |T^*|^2 \right\|.
		\end{eqnarray*}
We now extend the above inequality via Orlicz function.
 
	\begin{theorem} \label{th 5}
		Let $T \in \mathcal{B}(\mathcal{H}).$ Then for any Orlicz function $\phi,$
		\begin{eqnarray*}
			\phi\left(w^2(T)\right) &\leq& \frac{1}{2} \phi\left(\frac12w^2\left(|T| + i |T^*|\right)\right) + \frac{1}{4} \phi\left(w\left(|T||T^*|\right)\right) + \frac{1}{8} \left\|\phi(|T|^2) + \phi(|T^*|^2) \right\|.
		\end{eqnarray*}
	\end{theorem}
	\begin{proof}
		Let $x \in \mathcal{H}$ with $\|x\|=1.$ Using the convexity of the Orlicz function $\phi$, we have 
		\begin{eqnarray*}
			\phi\left(|\langle Tx, x \rangle|^2\right) &\leq& \phi \left( \langle |T|x, x \rangle \langle |T^*|x, x \rangle\right)~(\mbox{using Lemma \ref{mixed schwarz}})\\
			&\leq& \phi \left(\left( \frac{\langle |T|x, x \rangle + \langle |T^*|x, x \rangle}{2}\right)^2\right)\\
			&=& \phi \left(\frac{\langle |T|x, x \rangle ^2 + \langle |T^*|x, x \rangle ^2}{4} + \frac{\langle |T|x, x \rangle \langle x, |T^*|x \rangle }{2}\right) \\
			&\leq& \phi \left(\frac{1}{4}\left|\left\langle \left(|T| + i |T^*|\right)x, x \right\rangle\right|^2 + \frac{1}{4} \left(\||T|x\|\||T^*|x\| + |\langle |T|x, |T^*|x \rangle| \right) \right)\\
			&\leq& \frac{1}{2} \phi \left(\frac12\left|\left\langle \left(|T| + i |T^*|\right)x, x \right\rangle\right|^2\right) + \frac{1}{4} \phi \left(\||T|x\|\||T^*|x\|  \right)\\ && + \frac{1}{4} \phi \left(|\langle\left( |T^*| |T|\right)x, x \rangle|\right)
				\end{eqnarray*}
			\begin{eqnarray*}
			 &\leq& \frac{1}{2} \phi \left(\frac12 \left|\left\langle \left(|T| + i |T^*|\right)x, x \right\rangle\right|^2\right) + \frac{1}{4} \phi \left(\frac{\||T|x\|^2 + \||T^*|x \|^2}{2} \right)\\ && + \frac{1}{4} \phi \left(|\langle\left( |T^*| |T|\right)x, x \rangle|\right)\\
			&\leq&\frac{1}{2} \phi \left(\frac12 \left|\left\langle \left(|T| + i |T^*|\right)x, x \right\rangle\right|^2\right)  + \frac{1}{8} \left\langle \left(\phi(|T|^2) + \phi(|T^*|^2)\right)x, x \right\rangle \\ && +\frac{1}{4} \phi \left(|\langle\left( |T^*| |T|\right)x, x \rangle|\right)\\
		&\leq&	\frac{1}{2} \phi\left(\frac12 w^2\left(|T| + i |T^*|\right)\right) + \frac{1}{4} \phi\left(w\left(|T^*||T|\right)\right) + \frac{1}{8} \left\|\phi(|T|^2) + \phi(|T^*|^2) \right\|.
		\end{eqnarray*}
	\end{proof}

Also, from the proof of Theorem \ref{th 5}, we can show that   
\begin{cor}
    Let $T \in \mathcal{B}(\mathcal{H})$. Then for any Orlicz function $\phi$,
			\[\phi\left(w(T) \right) \leq \frac12 \left\|  \phi \left(|T|\right) + \phi \left(|T^*| \right)  \right\|.\]
   \end{cor}

Next bound is as follows.
   
	\begin{theorem} \label{th 6}
		Let $T \in \mathcal{B}(\mathcal{H}).$ Then for any Orlicz function $\phi,$ 
		\begin{eqnarray*}
			\phi\left(w^2(T)\right) &\leq& \frac{1}{2} \min \left\{ \phi\left(w\left(|T||T^*|\right)\right), \phi\left(w(T^2)\right)\right \} + \frac{1}{4} \left\|\phi(|T|^2) + \phi(|T^*|^2) \right\|.
		\end{eqnarray*}
	\end{theorem}
	
	\begin{proof}
		Let $x \in \mathcal{H}$ with $\|x\|=1.$ From the convexity of the Orlicz function $\phi,$ we get
		\begin{eqnarray*}
			\phi\left(|\langle Tx, x \rangle|^2\right)
   &=& \phi \left(\alpha\langle Tx, x \rangle|^2 + (1-\alpha) \langle Tx, x \rangle|^2 \right)\\
   &\leq& \alpha \phi \left(\langle Tx, x \rangle|^2\right) + (1-\alpha) \phi \left(\langle Tx, x \rangle|^2\right)\\
			&\leq& \alpha \phi \left(\langle |T|x, x \rangle \langle |T^*|x, x \rangle\right) + (1-\alpha) \phi \left(|\langle Tx, x \rangle  \langle x, T^*x \rangle| \right)\\  &&~~\,\,\,\,\,\,\,\,\,\,\,\,\,\,\,\,\,\,\,\,\,\,\,\,\,\,\,\,\,\,\,\,\,\,\,\,\,\,\,\,\,\,\,\,\,\,\,\,\,\,\,\,\,\,\,\,\,\,\,\,\,\,\,\,\,\,\,\,\,\,\,\,\,\,\,\,\,\,\,\,\,\,\,\,\,\,\,\,\,\,\,\,\,\,\,\,\,\,\,\,\,\,\,\,\,\,\,\,\,\,\,\,(\mbox{using Lemma \ref{mixed schwarz}})\\&\leq& \alpha \phi \left(\langle |T|x, x \rangle \langle x, |T^*|x \rangle\right) + (1-\alpha) \phi \left(\frac{|\langle Tx, T^*x \rangle| + \|Tx\|\|T^*x\|}{2}\right)\\  &&~~\,\,\,\,\,\,\,\,\,\,\,\,\,\,\,\,\,\,\,\,\,\,\,\,\,\,\,\,\,\,\,\,\,\,\,\,\,\,\,\,\,\,\,\,\,\,\,\,\,\,\,\,\,\,\,\,\,\,\,\,\,\,\,\,\,\,\,\,\,\,\,\,\,\,\,\,\,\,\,\,\,\,\,\,\,\,\,\,\,\,\,\,\,\,\,\,\,\,\,\,\,\,\,\,\,\,\,\,\,\,\,\,(\mbox{using Lemma \ref{buzano}})\\ &\leq& \alpha \phi \left(\frac{\left|\langle |T|x, |T^*|x \rangle \right| + \||T|x\| \||T^*|x\|}{2}\right) + \frac{1-\alpha}{2} \phi \left(|\langle T^2x, x \rangle|\right)\\ && + \frac{1-\alpha}{2} \phi \left(\frac{\|Tx\|^2 + \|T^*x\|^2}{2}\right)\\ 
			&\leq& \frac{\alpha}{2} \phi \left(|\langle \left(|T^*||T|\right)x, x \rangle| \right) + \frac{\alpha}{2} \phi \left(\frac{\||T|x\|^2 + \||T^*|x\|^2}{2}\right)\\ &&+ \frac{1-\alpha}{2} \phi \left(|\langle T^2x, x \rangle|\right) + \frac{1-\alpha}{2} \phi \left(\left\langle \left(\frac{|T|^2 + |T^*|^2}{2}\right)x, x \right \rangle \right)\\
			&\leq& \frac{\alpha}{2} \phi \left(|\langle |T^*||T|x, x \rangle| \right) + \frac{\alpha}{2} \phi \left(\left\langle \left(\frac{|T|^2 + |T^*|^2}{2}\right)x, x \right \rangle \right)\\ &&+ \frac{1-\alpha}{2} \phi \left(|\langle T^2x, x \rangle|\right) + \frac{1-\alpha}{2} \left\langle \left(\frac{\phi(|T|^2) + \phi(|T^*|^2)}{2}\right)x, x \right \rangle \\
			&\leq& \frac{\alpha}{2} \phi \left(w(|T^*| |T|)\right) + \frac{1}{4} \left\| \phi(|T|^2) + \phi (|T^*|^2) \right\|  + \frac{1-\alpha}{2} \phi \left(w(T^2)\right).
			\end{eqnarray*}
   Therefore, taking the supremum over $\|x\|=1,$ we obtain the desired bound.
	\end{proof}

\begin{remark}
    In particular, taking the Orlicz function $\phi(t) = t^r$, $r \geq 1$ in Theorem \ref{th 6}, we obtain that
 for $r\geq 1,$
		\begin{eqnarray}
			w^{2r}(T) &\leq&  \frac{1}{2} \min \left\{ w^r\left(|T||T^*|\right), w^r(T^2)\right \}  + \frac{1}{4} \left\||T|^{2r} + |T^*|^{2r} \right\|.
		\end{eqnarray}
		Clearly, this bound is stronger than the bound \cite[Corollary  2.6]{bulletin sci}, namely,
		$w^2(T) \leq \frac{1}{2} w \left( |T| |T^*|\right)+ \frac{1}{4} \left\||T|^2 + |T^*|^2 \right\|. $
\end{remark}
	
	From the proof of Theorem \ref{th 6}, we can also show that
\begin{theorem} \label{th 8}
		Let $T \in \mathcal{B}(\mathcal{H}).$ Then for any Orlicz function $\phi,$ 
		\begin{eqnarray*}
			\phi\left(w^2(T)\right) &\leq& \frac{1}{2} \min \left\{ \phi\left(w\left(|T||T^*|\right)\right), \phi\left(w(T^2)\right)\right \} + \frac{1}{2} \phi \left( \frac12  \left\||T|^2 + |T^*|^2 \right\|\right).
		\end{eqnarray*}
	\end{theorem}
	
	By choosing the Orlicz function $\phi(t)=e^t-1$ in Theorem \ref{th 8}, we obtain
	\begin{cor}
		Let $T \in \mathcal{B}(\mathcal{H}).$ Then 
		\begin{eqnarray}\label{1-1}
		    {w^2(T)} \leq   \log \left[ \frac{1}{2} e^{w(T^2)} + \frac12  e^{ \frac{ \|T^*T+TT^*  \|}{2} }\right] &\leq& \frac12 \|T^*T+TT^*  \|
		\end{eqnarray} 
 and 
 \begin{eqnarray}\label{1-2}
		    {w^2(T)} \leq   \log \left[ \frac{1}{2} e^{w(|T||T^*|)} + \frac12  e^{ \frac{ \|T^*T+TT^*  \|}{2} }\right] &\leq& \frac12 \|T^*T+TT^*  \|.
		\end{eqnarray} 
	\end{cor}

Also applying Proposition \ref{Orlicz buzano}, we obtain the following numerical radius bound.
\begin{cor}\label{prop-1}
 Let $T \in \mathcal{B}(\mathcal{H}).$ Then
 \begin{eqnarray}
     w^4(T) \leq \log \left(\frac12 e^{w^2(T^2)}+\frac12e^{\frac{ \left\|  |T|^4+|T^*|^4\right\|}{2} }\right) &\leq& \frac{ \left\|  |T|^4+|T^*|^4\right\|}{2}.
 \end{eqnarray}
\end{cor}
\begin{proof}
Let $x\in \mathcal{H}$ with $\|x\|=1.$
    Replacing  $x$ by $Tx$, $y$ by $T^*x$ and $e$ by $x$ in the inequality \eqref{0--1}, we obtain 
    \begin{eqnarray*}
        |\langle Tx,x\rangle|^4&\leq& \log \left( \frac12 e^{\|Tx\|^2 \|T^*x\|^2} +\frac12 e^{|\langle Tx, T^*x\rangle|^2}\right)\\
        &\leq& \log \left( \frac12 e^{\frac{\|Tx\|^4 +\|T^*x\|^4}{2}} +\frac12 e^{|\langle T^2x, x\rangle|^2}\right)\\
        &\leq& \log \left(\frac12 e^{w^2(T^2)}+\frac12e^{\frac{ \left\|  |T|^4+|T^*|^4\right\|}{2} }\right)
        \leq \frac{ \left\|  |T|^4+|T^*|^4\right\|}{2}.
    \end{eqnarray*}
    Considering the supremum over $\|x\|=1$, we get the desired bounds.
\end{proof}

In \cite{pintu arxiv}, Bhunia proved that $w^n(T)  \leq \frac{1}{2^{n-1}} w(T^n) +  \sum_{k=1}^{n-1} \frac{1}{2^k} \|T^k\|\|T\|^{n-k} \leq \frac{1}{2^{n-1}} w(T^n) +  \left( 1-\frac{1}{2^{n-1}}\right) {\|T\|^n} \quad \text{for every $n\in \mathbb{N}$}.$
We obtain a generalisation of the above inequalities via Orlicz function.
 For this purpose we need the following lemma.
	
	\begin{lemma}\label{extension buzano} \cite[Lemma 2.1]{pintu arxiv}
		Let $x_1, x_2, x_3 \dots x_n, e \in \mathcal{H}$ with $\|e\|=1.$ Then 
		\[\left|\prod_{k=1}^{n}\langle x_k, e \rangle \right| \leq  \frac{\left|\langle x_1, x_2 \rangle \prod_{k=3}^n \langle x_k, e \rangle \right| + \prod_{k=1}^{n}  \|x_k\|}{2}.\]
	\end{lemma} 
	
	\begin{theorem} \label{th 7}
		Let $T \in \mathcal{H}.$ Then for any Orlicz function $\phi$ and for every $n \in \mathbb{N},$
		\begin{eqnarray*}
		    \phi \left(w^n(T)\right)  &\leq& \frac{1}{2^{n-1}} \phi \left(w(T^n)\right) +  \sum_{k=1}^{n-1} \frac{1}{2^k} \phi \left(\|T^k\|\|T\|^{n-k}\right)\\
  &\leq &\frac{1}{2^{n-1}} \phi \left(w(T^n)\right) +  \left( 1-\frac{1}{2^{n-1}}\right) \phi \left({\|T\|^n}\right).
		\end{eqnarray*}
  \end{theorem}

	\begin{proof}
		Let $x \in \mathcal{H}$ with $\|x\| =1.$ Then by the convex property of $\phi$, we get
 	\begin{eqnarray*}
			\phi \left(|\langle Tx, x \rangle|^n\right) &=& \phi \left(\left|\langle Tx, x \rangle \langle T^*x, x \rangle \langle T^*x, x \rangle^{n-2}\right| \right)\\ 
   &\leq& \phi \left(\frac{\left|\langle Tx, T^*x \rangle \langle T^*x ,x \rangle^{n-2}\right| + \|Tx\|\|T^*x \|^{n-1}}{2}\right)\\ &&\,\,\,\,\,\,\,\,\,\,\,\,\,\,\,\,\,\,\,\,\,\,\,\,\,\,\,\,\,\,\,\,\,\,\,\,\,\,\,\,\,\,\,\,\,\,\,\,\,\,\,\,\,\,\,\,(\mbox{using Lemma \ref{extension buzano}})\\ 
   &\leq& \frac{1}{2} \phi \left(\left|\langle T^2x, x \rangle \langle T^*x, x \rangle^{n-2}\right|\right)+ \frac{1}{2} \phi \left(\|Tx\|\|T^*x\|^{n-1}\right)\\ 
   &=& \frac{1}{2} \phi \left(\left|\langle T^2x, x \rangle \langle T^*x, x \rangle \langle T^*x, x \rangle^{n-3}\right| \right) + \frac{1}{2} \phi \left(\|Tx\|\|T^*x\|^{n-1}\right) \\ 
   &\leq& \frac{1}{2} \phi \left(\frac{\left|\langle T^2x, T^*x \rangle \langle T^*x, x \rangle^{n-3}\right| + \|T^2x\| \|T^*x\|^{n-2}}{2}\right)\\&&+ \frac{1}{2} \phi \left(\|Tx\|\|T^*x\|^{n-1}\right)\\ &&\,\,\,\,\,\,\,\,\,\,\,\,\,\,\,\,\,\,\,\,\,\,\,\,\,\,\,\,\,\,\,\,\,\,\,\,\,\,\,\,\,\,\,\,\,\,\,\,\,\,\,\,\,\,\,\,(\mbox{using Lemma \ref{extension buzano}})\\ &\leq& \frac{1}{4} \phi \left(\left|\langle T^3x, x \rangle \langle T^*x, x \rangle \langle T^*x, x \rangle^{n-4}\right| \right) + \frac{1}{4} \phi \left(\|T^2x\|\|T^*x\|^{n-2}\right) \\ && + \frac{1}{2} \phi \left(\|Tx\|\|T^*x\|^{n-1}\right)\\
			&\leq& \frac{1}{4} \phi \left(\frac{\left|\langle T^3x, T^*x \rangle  \langle T^*x, x\rangle^{n-4}\right| + \|T^3x\| \|T^*x\|^{n-3}}{2}\right)\\ && +  \frac{1}{4} \phi \left(\|T^2x\|\|T^*x\|^{n-2}\right)  + \frac{1}{2} \phi \left(\|Tx\|\|T^*x\|^{n-1}\right)\\
			&\leq& \frac{1}{8} \phi \left(\left|\langle T^4x, x \rangle \langle T^*x, x \rangle^{n-4}\right|\right) + \frac{1}{8} \phi \left(\|T^3x\|\|T^*x\|^{n-3}\right)  \\  && +  \frac{1}{4} \phi \left(\|T^2x\|\|T^*x\|^{n-2}\right)  + \frac{1}{2} \phi \left(\|Tx\|\|T^*x\|^{n-1}\right).
		\end{eqnarray*}
		Repeating this process, we obtain
		\begin{eqnarray*}
		\phi \left(|\langle Tx, x \rangle|^n\right) &\leq &\frac{1}{2^{n-1}} \phi \left(|\langle T^n x , x\rangle|\right) + \sum_{k=1}^{n-1} \frac{1}{2^k} \phi \left(\|T^kx\|\|T^*x\|^{n-k}\right)\\
		&\leq& \frac{1}{2^{n-1}} \phi \left(w(T^n)\right) +  \sum_{k=1}^{n-1} \frac{1}{2^k} \phi \left(\|T^k\|\|T\|^{n-k}\right).
		\end{eqnarray*}
		By taking the supremum over all unit vectors $x,$ we get the required inequalities.
	\end{proof}


 By choosing the Orlicz function $\phi(t)=e^t-1$ in Theorem \ref{th 7}, we obtain
	\begin{cor}
		Let $T \in \mathcal{B}(\mathcal{H}).$ Then for every $n (\geq 2) \in \mathbb{N}$,
		\begin{eqnarray}\label{nil}
		    {w(T)} \leq  \sqrt[n]{ \log \left[ \frac{1}{2^{n-1}} e^{w(T^n)} + \left( 1-\frac{1}{2^{n-1}}\right) e^{\|T\|^n}\right]} &\leq&  \|T\|.
		\end{eqnarray} 
  In particular, for n=2,
  \begin{eqnarray}\label{N-222}
		    {w(T)} \leq \sqrt{\log \left( \frac{1}{2} e^{w(T^2)} +\frac12  e^{\|T\|^2}  \right)} &\leq& \|T\|.
		\end{eqnarray} 
	\end{cor}

 Following result gives an estimation for the  numerical radius of nilpotent operators.

 \begin{cor}
     Let $T \in \mathcal{B}(\mathcal{H})$. If $T^n=0$ for some $n (\geq 2) \in \mathbb{N}$, then
		\begin{eqnarray*}
		    {w(T)} \leq  \sqrt[n]{  \log \left[ \frac{1}{2^{n-1}}  + \left( 1-\frac{1}{2^{n-1}}\right) e \right] } \, \|T\|.
		\end{eqnarray*} 
 \end{cor}
	\begin{proof}
	    Consider  $\|T\|=1.$ Then from \eqref{nil}, we get $ {w^n(T)} \leq    \log \left[ \frac{1}{2^{n-1}}  + \left( 1-\frac{1}{2^{n-1}}\right) e \right].$ 
  This completes the proof.
	\end{proof}

	\bigskip
	
	\noindent \textbf{Declarations:} 
	\textbf{Conflict of Interest:} The authors declare that there is no conflict of interest.\\

	\bibliographystyle{amsplain}
			
	\end{document}